\documentclass{amsart}
\usepackage{color,xspace}
\usepackage[normalem]{ulem}
\usepackage{mathrsfs}
\usepackage{stmaryrd}
\usepackage{graphicx}
\usepackage{url}

\newcommand\RE{\mathbb{R}}
\newcommand\CO{\mathbb{C}}

\newcommand\sfzero{\mathsf{0}}
\newcommand\sfA{\mathsf{A}}
\newcommand\sfB{\mathsf{B}}
\newcommand\sfC{\mathsf{C}}
\newcommand\sfD{\mathsf{D}}
\newcommand\sfE{\mathsf{E}}
\newcommand\sfF{\mathsf{F}}
\newcommand\sfG{\mathsf{F}}
\newcommand\sfH{\mathsf{F}}

\newcommand\sfx{\mathsf{x}}
\newcommand\sfy{\mathsf{y}}
\newcommand\sfz{\mathsf{z}}

\newcommand\bfsigma{\underline{\boldsymbol{\sigma}}}
\newcommand\bftau{\underline{\boldsymbol{\tau}}}
\newcommand\Id{\underline{\mathbf{I}}}
\newcommand\bfu{\mathbf{u}}
\newcommand\bfv{\mathbf{v}}
\newcommand\bff{\mathbf{f}}
\newcommand\bfp{\mathbf{p}}
\newcommand\bfq{\mathbf{q}}
\newcommand\bfs{\underline{\mathbf{s}}}
\newcommand\bft{\underline{\mathbf{t}}}
\newcommand\bfx{\mathbf{x}}
\newcommand\cA{\mathcal{A}}
\newcommand\symgrad{\operatorname{\underline{\boldsymbol{\varepsilon}}}}
\newcommand\tr{\operatorname{\mathrm{tr}}}
\newcommand\as{\operatorname{\mathrm{as}}}
\newcommand\bfX{\underline{\mathbf{X}}}
\newcommand\bfchi{\underline{\boldsymbol{\chi}}}
\newcommand\bfpsi{\boldsymbol{\psi}}
\newcommand\bfphi{\boldsymbol{\varphi}}
\renewcommand\div{\operatorname{\mathbf{div}}}
\newcommand\grad{\operatorname{\underline{\boldsymbol{\nabla}}}}
\newcommand\Hdiv{\mathbf{H}(\div;\Omega)}

\newcommand\bfzero{\mathbf{0}}
\newcommand\bfn{\mathbf{n}}
\newcommand\curl{\operatorname{\mathbf{curl}}}

\newcommand\rank{\operatorname{\mathrm{rank}}}

\newcommand\DB{\relax}
\newcommand\BD{\relax}

\newtheorem{theorem}{Theorem}

\newtheorem{proposition}[theorem]{Proposition}

\theoremstyle{remark}
\newtheorem{remark}{Remark}

\begin{document}

\title[Least-squares for linear elasticity eigenvalue problem]{Least-squares
formulations for eigenvalue problems associated with linear elasticity}

\author{Fleurianne Bertrand}
\address{Humboldt-Universit\"at zu Berlin, Germany and King Abdullah
University of Science and Technology, Saudi Arabia}
\curraddr{}
\email{}
\thanks{}

\author{Daniele Boffi}
\address{King Abdullah University of Science and Technology, Saudi Arabia, and
University of Pavia, Italy}
\curraddr{}
\email{}
\thanks{}

\date{}

\begin{abstract}

We study the approximation of the spectrum of least-squares operators arising
from linear elasticity.
We consider a two-field (stress/displacement) and a three-field
(stress/displacement/vorticity) formulation; other formulations might be
analyzed with similar techniques.

We prove a priori estimates and we confirm the theoretical results with simple
two-dimensional numerical experiments.

\end{abstract}

\maketitle

\section{Introduction}

In this paper we continue the investigations started in~\cite{ls4eig} about
the spectral properties of operators associated with finite element
least-squares formulations. While~\cite{ls4eig} deals with the Poisson
eigenvalue problem, here we consider linear elasticity in the
stress/displacement formulation. We discuss a two-field least-squares
formulation introduced in~\cite{CaiStarke2004} and a three-field least-squares
formulation studied in~\cite{BertrandCaiPark2019}.

We show that under natural approximation properties of the involved finite
element spaces the discretization of the eigenvalue problems converges
optimally by using the standard Babu\v ska--Osborn theory~\cite{BaOs,acta}.

An important difference with respect to the case of the Poisson equation, is
that for elasticity problem the resulting formulation is not symmetric. We
will see that in most cases the eigenvalues are nevertheless real, but there
are situations in which complex solutions are found.

As for the analysis presented in~\cite{ls4eig}, it should be clear that the
aim of this study is not to present a competitive method for the computation
of elasticity eigenmodes. However, knowing the behavior of the spectrum of
least-squares operators can be useful for several reasons; for instance,
people interested in implementing least-squares approximations for transient
problems, could benefit from such information.

Some numerical tests confirm our theoretical results.

\section{Problem setting}

Given a polytopal domain $\Omega\in\RE^d$ ($d=2,3$) with boundary
$\partial\Omega$ divided into two parts $\Gamma_D$ and $\Gamma_N$, we consider
the stress/displacement formulation of linear elasticity as a system of first
order as follows: find a symmetric $d$-by-$d$ stress tensor $\bfsigma$ and a
displacement vectorfield $\bfu$ such that
\begin{equation}
\aligned
&\cA\bfsigma-\symgrad(\bfu)=0&&\text{in }\Omega\\
&\div\bfsigma=-\bff&&\text{in }\Omega\\
&\bfu=\bfzero&&\text{on }\Gamma_D\\
&\bfsigma\bfn=\bfzero&&\text{on }\Gamma_N,
\endaligned
\label{eq:elassource}
\end{equation}
where the compliance tensor is defined in terms of the Lam\'e constants $\mu$
and $\lambda$
\[
\cA\bftau=\frac{1}{2\mu}\left(\bftau-\frac{\lambda}{2\mu+d\lambda}\tr(\bftau)\Id\right),
\]
the symmetric gradient is defined as
\[
\symgrad(\bfv)=\frac12\left(\grad\bfv+(\grad\bfv)^\top\right),
\]
and $\bfn$ denotes the outward unit normal vector to $\Gamma_N$.

The formulation we are discussing has the important property to be robust in
the incompressible limit. Other formulations could be studied as well. We
refer the interested reader, for instance, to the introduction
of~\cite{BertrandCaiPark2019} for an historical perspective.

A least-squares formulation of~\eqref{eq:elassource} was considered
in~\cite{CaiStarke2004} by minimizing the functional
\[
\mathcal{F}(\bftau,\bfv;\bff)=\|\cA\bftau-\symgrad(\bfv)\|_0^2+\|\div\bftau+\bff\|_0^2
\]
in $\bfX_N\times H_{0,D}^1(\Omega)^d$, with
\[
\bfX=
\begin{cases}
\Hdiv^d&\text{if }\Gamma_N\ne\emptyset\\
\left\{\bftau\in\Hdiv^d:\int_\Omega\tr(\bftau)\,d\bfx=0\right\}&\text{if }\Gamma_N=\emptyset
\end{cases}
\]
and $\bfX_N$ the subset of $\bfX$ corresponding to the boundary condition
$\bftau\bfn=\bfzero$ on $\Gamma_N$.

\begin{remark}
This formulation does not require a symmetry constraint in the definition of
$\bfX_N$ since the constitutive equation implies automatically that the
solution satisfies $\bfsigma=\bfsigma^\top$. We shall refer to this
formulation as the \emph{two-field} formulation.
\end{remark}

Other two formulations have been introduced in~\cite{BertrandCaiPark2019}
which make use of three and four fields, respectively, by the introduction of
the vorticity and the pressure as new unknowns. We describe the three-field
formulation, similar considerations could be derived for the four-field
formulation.

The three-field formulation seeks a minimizer of the functional
\[
\mathcal{G}(\bftau,\bfv,\bfphi;\bff)=\|\cA\bftau-\grad\bfv+(-1)^d\bfchi\bfphi\|_0^2+
\|\div\bftau+\bff\|_0^2+\|\as\bftau\|_0^2
\]
in $\bfX_N\times H_{0,D}^1(\Omega)^d\times\bar{L}(\Omega)^{2d-3}$ with
\[
\bar{L}^2(\Omega)=
\begin{cases}
L^2(\Omega)&\text{if }\Gamma_N\ne\emptyset\\
\left\{\bfphi\in L^2(\Omega):\int_\Omega\bfphi\,d\bfx=0\right\}&\text{if
}\Gamma_N=\emptyset,
\end{cases}
\]
where
\[
\bfchi=
\begin{cases}
\begin{pmatrix}
0&-1\\1&0
\end{pmatrix}
&\text{if }d=2\\
(\bfchi_1,\bfchi_2,\bfchi_3)&\text{if }d=3
\end{cases}
\]
with
\[
\bfchi_1=
\begin{pmatrix}
0&0&0\\0&0&-1\\0&1&0
\end{pmatrix}
\quad
\bfchi_2=
\begin{pmatrix}
0&0&1\\0&0&0\\-1&0&0
\end{pmatrix}
\quad
\bfchi_3=
\begin{pmatrix}
0&-1&0\\1&0&0\\0&0&0
\end{pmatrix}
\]
and where $\as\bftau=\left(\bftau-\bftau^\top\right)/2$ is the skew-symmetric part of
$\bftau$.

We are interested in the eigenvalue problem corresponding
to~\eqref{eq:elassource}, that is, we are looking for $\omega\in\RE$ such that
for non-vanishing $\bfu$ and for some $\bfsigma$ it holds
\begin{equation}
\aligned
&\cA\bfsigma-\symgrad(\bfu)=0&&\text{in }\Omega\\
&\div\bfsigma=-\omega\bfu&&\text{in }\Omega\\
&\bfu=\bfzero&&\text{on }\Gamma_D\\
&\bfsigma\bfn=\bfzero&&\text{on }\Gamma_N.
\endaligned
\label{eq:elaseig}
\end{equation}

Thanks to the regularity properties of the solution of~\eqref{eq:elassource},
the eigenvalue problem~\eqref{eq:elaseig} is compact so that its eigenvalues
form an increasing sequence
\[
0<\omega_1\le\omega_2\le\dots\le\omega_i\le\cdots
\]
and the eigenspaces are finite dimensional. As usual, we repeat the
eigenvalues according to their multiplicity, so that each eigenvalue
$\omega_i$ corresponds to a one-dimensional eigenspace $E_i$.

In order to approximate the eigenmodes, we are going to generalize the ideas
of~\cite{ls4eig} to the two- and three-field formulations presented above. In
particular, we are not writing directly a least-squares formulation of the
eigenvalue problem (which would lead to a non-linear problem), but we study
the spectrum of the operators associated with the least-squares \emph{source}
formulations.

\subsection{Eigenvalue problem associated with the two-field formulation}

The minimization of the functional $\mathcal{F}(\bftau,\bfv,\bff)$ gives rise
to the following variational formulation: find $\bfsigma\in\bfX_N$ and
$\bfu\in H^1_{0,D}(\Omega)^d$ such that
\begin{equation}
\left\{
\aligned
&(\cA\bfsigma,\cA\bftau)+(\div\bfsigma,\div\bftau)
-(\cA\bftau,\symgrad(\bfu))=-(\bff,\div\bftau)&&
\forall\bftau\in\bfX_N\\
&-(\cA\bfsigma,\symgrad(\bfv))+(\symgrad(\bfu),\symgrad(\bfv))=0&&
\forall\bfv\in H^1_{0,D}(\Omega)^d.
\endaligned
\right.
\label{eq:varsourceF}
\end{equation}

The corresponding eigenvalue problems is obtained after replacing $\bff$ with
$\omega\bfu$, that is: find $\omega\in\CO$ such that for a non-vanishing
$\bfu\in H^1_{0,D}(\Omega)^d$ and for some $\bfsigma\in\bfX_N$ we have
\begin{equation}
\left\{
\aligned
&(\cA\bfsigma,\cA\bftau)+(\div\bfsigma,\div\bftau)
-(\cA\bftau,\symgrad(\bfu))=-\omega(\bfu,\div\bftau)&&
\forall\bftau\in\bfX_N\\
&-(\cA\bfsigma,\symgrad(\bfv))+(\symgrad(\bfu),\symgrad(\bfv))=0&&
\forall\bfv\in H^1_{0,D}(\Omega)^d.
\endaligned
\right.
\label{eq:vareigF}
\end{equation}

The structure of the eigenvalue problem~\eqref{eq:vareigF} in terms of
operators is similar to the one described in~\cite{ls4eig} in the case of the
FOSLS formulation for the Poisson equation. Namely, by natural definition of
the operators, we are led to the following form:
\begin{equation}
\left(
\begin{matrix}
\sfA&\sfB^\top\\
\sfB&\sfC
\end{matrix}
\right)
\left(
\begin{matrix}
\sfx\\\sfy
\end{matrix}
\right)=
\omega\left(
\begin{matrix}
\sfzero&\sfD\\
\sfzero&\sfzero
\end{matrix}
\right)
\left(
\begin{matrix}
\sfx\\\sfy
\end{matrix}
\right).
\label{eq:Fmatrix}
\end{equation}

One important difference between this formulation and the one presented
in~\cite{ls4eig} for the Poisson equation is that in our case the operators
$\sfB^\top$ and $-\sfD$ are not the same. It follows that it is not possible
to show in general that~\eqref{eq:Fmatrix} corresponds to a symmetric problem.
This fact has important consequences for the numerical approximation. We will
see in Section~\ref{se:num} that most of the computed eigenvalues are
real, but there are exceptions.

\subsection{Eigenvalue problem associated with the three-field formulation}

The variational formulation associated with the minimization of the functional
$\mathcal{G}$ is obtained by seeking $\bfsigma\in\bfX_N$,
$\bfu\in H^1_{0,D}(\Omega)^d$, and $\bfpsi\in\bar{L}^2(\Omega)$ such that
\begin{equation}
\left\{
\aligned
&(\cA\bfsigma,\cA\bftau)+(\div\bfsigma,\div\bftau)+(\as(\bfsigma),\as(\bftau))\\
&\qquad-(\cA\bftau,\symgrad(\bfu))+(-1)^d(A\bftau,\bfchi\bfpsi)=-(\bff,\div\bftau)&&
\forall\bftau\in\bfX_N\\
&-(\cA\bfsigma,\symgrad(\bfv))+(\symgrad(\bfu),\symgrad(\bfv))
-(-1)^d(\bfchi\bfpsi,\grad\bfv)=0&&\forall\bfv\in H^1_{0,D}(\Omega)^d\\
&(-1)^d(A\bfsigma,\bfchi\bfphi)-(-1)^d(\bfchi\bfphi,\grad\bfu)
+(\bfchi\bfpsi,\bfchi\bfphi)=0&&\forall\bfphi\in\bar{L}^2(\Omega).
\endaligned
\right.
\label{eq:varsourceG}
\end{equation}

Also in this case we consider the eigenvalue problem by replacing $\bff$ with
$\omega\bfu$ in the right hand side. The problem reads: find $\omega\in\CO$
such that for a non-vanishing $\bfu\in H^1_{0,D}(\Omega)^d$ and for some
$\bfsigma\in\bfX_N$ and $\bfpsi\in\bar{L}^2(\Omega)$ it holds
\begin{equation}
\left\{
\aligned
&(\cA\bfsigma,\cA\bftau)+(\div\bfsigma,\div\bftau)+(\as(\bfsigma),\as(\bftau))\\
&\qquad-(\cA\bftau,\symgrad(\bfu))+(-1)^d(A\bftau,\bfchi\bfpsi)=-\omega(\bfu,\div\bftau)&&
\forall\bftau\in\bfX_N\\
&-(\cA\bfsigma,\symgrad(\bfv))+(\symgrad(\bfu),\symgrad(\bfv))
-(-1)^d(\bfchi\bfpsi,\grad\bfv)=0&&\forall\bfv\in H^1_{0,D}(\Omega)^d\\
&(-1)^d(A\bfsigma,\bfchi\bfphi)-(-1)^d(\bfchi\bfphi,\grad\bfu)
+(\bfchi\bfpsi,\bfchi\bfphi)=0&&\forall\bfphi\in\bar{L}^2(\Omega).
\endaligned
\right.
\label{eq:vareigG}
\end{equation}

The operator form of the eigenvalue problem~\eqref{eq:vareigG} involves
$3$-by-$3$ block operators as follows:
\begin{equation}
\left(
\begin{matrix}
\sfA&\sfB^\top&\sfC^\top\\
\sfB&\sfD&\sfE^\top\\
\sfC&\sfE&\sfG
\end{matrix}
\right)
\left(
\begin{matrix}
\sfx\\\sfy\\\sfz
\end{matrix}
\right)=
\omega\left(
\begin{matrix}
\sfzero&\sfH&\sfzero\\
\sfzero&\sfzero&\sfzero\\
\sfzero&\sfzero&\sfzero
\end{matrix}
\right)
\left(
\begin{matrix}
\sfx\\\sfy\\\sfz
\end{matrix}
\right).
\label{eq:Gmatrix}
\end{equation}

Also in this case the system is not symmetric and the numerical approximation
presented in Section~\ref{se:num} will show that some computed eigenvalues may
be complex.

\section{Numerical approximation}
\label{se:numap}

As it is apparent from formulations~\eqref{eq:vareigF} and~\eqref{eq:vareigG},
the numerical approximation will lead to generalized eigenvalue problems of
the form $\mathcal{A}x=\omega\mathcal{B}\DB x\BD$ where the matrix $\mathcal{B}$ is
singular. From the algebraic point of view, as we observed in~\cite{ls4eig},
the solution of this problem satisfies the following properties.

\begin{enumerate}

\item If the matrix $\mathcal{B}$ is invertible then our system is
equivalent to the standard eigenvalue problem
$\mathcal{B}^{-1}\mathcal{A}x=\omega x$.

\item If $\mathcal{K}=\ker\mathcal{A}\cap\ker\mathcal{B}$ is not trivial then
the eigenvalue problem is degenerate and vectors in $\mathcal{K}$ do not
correspond to any eigenvalue of our original system.

\item If the matrix $\mathcal{B}$ has a non-trivial kernel which does not
contain any nonzero vector of $\ker(\mathcal{A})$ then it is conventionally
assumed that our system has an eigenvalue $\omega=\infty$ with eigenspace
equal to $\ker(\mathcal{B})$.

\item If $\mathcal{B}$ is singular and $\mathcal{A}$ is not (which is the most
common situation in our framework) then it may be convenient to switch the
roles of the two matrices and to consider the problem
\[
\mathcal{B}x=\gamma\mathcal{A}x.
\]
Then $(\gamma,x)$ with $\gamma=0$ corresponds to the eigenmode $(\infty,x)$
of the original system; the remaining eigenmodes are $(\omega,x)$ with
$\omega=1/\gamma$.

\end{enumerate}

\DB
In the examples we will discuss, we are not going to deal with Case~(3)
although that situation would open intriguing and interesting new scenarios,
similar to what was presented, for instance, in~\cite{ref1}. More aspects of
the numerical implementation will be mentioned in Section~\ref{se:num}.
\BD

\subsection{Approximation of the two-field formulation}

Given finite dimensional subspaces $\Sigma_h\subset\bfX_N$ and $U_h\subset
H^1_{0,D}(\Omega)^d$, the Galerkin approximation of~\eqref{eq:vareigF} reads:
find $\omega_h\in\CO$ such that for a non-vanishing $\bfu_h\in U_h$ and for
some $\bfsigma_h\in\Sigma_h$ we have
\begin{equation}
\left\{
\aligned
&(\cA\bfsigma_h,\cA\bftau)+(\div\bfsigma_h,\div\bftau)
-(\cA\bftau,\symgrad(\bfu_h))=-\omega_h(\bfu_h,\div\bftau)&&
\forall\bftau\in\Sigma_h\\
&-(\cA\bfsigma_h,\symgrad(\bfv))+(\symgrad(\bfu_h),\symgrad(\bfv))=0&&
\forall\bfv\in U_h.
\endaligned
\right.
\label{eq:vareigFh}
\end{equation}

The structure of this problem is analogous of~\eqref{eq:Fmatrix} with the
natural definition of the involved matrices. The following proposition is the
analogue of~\cite[Prop.~6]{ls4eig} and characterizes the number of significant
eigenvalues of~\eqref{eq:vareigFh}.

\begin{proposition}
The solution of the generalized eigenvalue problem associated with the
formulation~\eqref{eq:vareigFh} includes $\omega=\infty$ with multiplicity
equal to $\dim(\Sigma_h)+\dim(\ker(\sfD))$ and a number of complex eigenvalues
(counted with their multiplicity) equal to $\rank(\sfD)$.
\label{pr:charF}
\end{proposition}

\begin{remark}
Since the eigenvalue problem stated in~\eqref{eq:vareigFh} considers
eigenfunctions with $\bfu_h\ne\bfzero$, the total number of eigenvalues of the
problem we are interested in, is equal to $\dim U_h$; the corresponding values
are the $\rank(\sfD)$ complex solutions and possibly $\omega=\infty$ if $\sfD$
is not full rank.
\label{re:full}
\end{remark}

\subsection{Approximation of the three-field formulation}

Given finite dimensional subspaces $\Sigma_h\subset\bfX_N$, $U_h\subset
H^1_{0,D}(\Omega)^d$, and $\Phi_h\subset\bar{L}^2(\Omega)$, the Galerkin
approximation of~\eqref{eq:vareigG} is: find $\omega_h\in\CO$ such that for a
non-vanishing $\bfu_h\in U_h$ and for some $\bfsigma_h\in\Sigma_h$ and
$\bfpsi_h\in\Phi_h$ we have
\begin{equation}
\left\{
\aligned
&(\cA\bfsigma_h,\cA\bftau)+(\div\bfsigma_h,\div\bftau)+(\as(\bfsigma_h),\as(\bftau))\\
&\qquad-(\cA\bftau,\symgrad(\bfu_h))+(-1)^d(A\bftau,\bfchi\bfpsi_h)=-\omega_h(\bfu_h,\div\bftau)&&
\forall\bftau\in\Sigma_h\\
&-(\cA\bfsigma_h,\symgrad(\bfv))+(\symgrad(\bfu_h),\symgrad(\bfv))
-(-1)^d(\bfchi\bfpsi_h,\grad\bfv)=0&&\forall\bfv\in U_h\\
&(-1)^d(A\bfsigma_h,\bfchi\bfphi)-(-1)^d(\bfchi\bfphi,\grad\bfu_h)
+(\bfchi\bfpsi_h,\bfchi\bfphi)=0&&\forall\bfphi\in\Phi_h.
\endaligned
\right.
\label{eq:vareigGh}
\end{equation}

The matrix structure of this problem corresponds to the one
of~\eqref{eq:vareigG} and the following proposition, in analogy of
Proposition~\ref{pr:charF}, gives the characterization of the discrete
eigenvalues.

\begin{proposition}
The solution of the generalized eigenvalue problem associated with the
formulation~\eqref{eq:vareigGh} consists of $\omega=\infty$ and a number of
complex eigenvalues (counted with their multiplicity) equal to $\rank(\sfF)$.
\end{proposition}

\begin{remark}
As in Remark~\ref{re:full}, the eigenvalues of formulation~\eqref{eq:vareigGh}
are $\rank(\sfF)$ complex values and possibly $\omega=\infty$ if $\sfF$ is not
full rank.
\end{remark}

\section{A priori error estimates}

It is quite standard to obtain a priori error estimates for eigenvalue
problems by studying the convergence of the discrete solution operator towards
the continuous one~\cite{acta}. In our framework, this boils down to showing
$L^2(\Omega)$-estimates for the discretization of~\eqref{eq:varsourceF}
and~\eqref{eq:varsourceG}, respectively, when the right hand side $\bff$ is in
$L^2(\Omega)^d$.
For brevity, we omit the details and we refer the interested reader
to~\cite{ls4eig}.

\subsection{$L^2(\Omega)$-estimate for the two-field formulation}

The $L^2(\Omega)$-estimate for the two-field
formulation~\eqref{eq:varsourceF} is the natural generalization of what has
been presented in~\cite{ls4eig} in the case of the Poisson problem. The
original idea comes from~\cite[Sec.~7]{vecquad} (see also~\cite{CaiKu}).

Let $(\bfsigma,\bfu)\in\bfX_N\times H^1_{0,D}(\Omega)^2$
solve~\eqref{eq:varsourceF} and $(\bfsigma_h,\bfu_h)\in\Sigma_h\times U_h$
solve the corresponding discrete problem; we consider the dual problem: find
$\bfs\in\bfX_N$ and $\bfp\in H^1_{0,D}(\Omega)^d$ such that
\begin{equation}
\left\{
\aligned
&(\cA\bfs,\cA\bft)+(\div\bfs,\div\bft)
-(\cA\bft,\symgrad(\bfp))=0&&\forall\bft\in\bfX_N\\
&-(\cA\bfs,\symgrad(\bfq))+(\symgrad(\bfp),\symgrad(\bfq))=
(\bfu-\bfu_h,\bfq)&&\forall\bfq\in H^1_{0,D}(\Omega)^d.
\endaligned
\right.
\label{eq:dualF}
\end{equation}
Taking as test functions $\bft=\bfsigma-\bfsigma_h$ and $\bfq=\bfu-\bfu_h$ we
have
\[
\aligned
\|\bfu-\bfu_h\|_0^2&=(\cA\bfs,\cA(\bfsigma-\bfsigma_h))
+(\div\bfs,\div(\bfsigma-\bfsigma_h))\\
&\quad-(\symgrad(\bfp),\cA(\bfsigma-\bfsigma_h))
-(\cA\bfs,\symgrad(\bfu-\bfu_h))\\
&\quad+(\symgrad(\bfp),\symgrad(\bfu-\bfu_h))\\
&=(\cA(\bfs-\bftau_h),\cA(\bfsigma-\bfsigma_h))
+(\div(\bfs-\bftau_h),\div(\bfsigma-\bfsigma_h))\\
&\quad-(\symgrad(\bfp-\bfv_h),\cA(\bfsigma-\bfsigma_h))
-(\cA(\bfs-\bftau_h),\symgrad(\bfu-\bfu_h))\\
&\quad+(\symgrad(\bfp-\bfv_h),\symgrad(\bfu-\bfu_h))
\endaligned
\]
for all $\bftau_h\in\Sigma_h$ and $\bfv_h\in U_h$, where we used the error
equations corresponding to~\eqref{eq:varsourceF}. It follows
\[
\aligned
\|\bfu-\bfu_h\|_0^2&\le C\left(\|\bfs-\bftau_h\|_{\div}+\|\bfp-\bfv_h\|_1\right)
\left(\|\bfsigma-\bfsigma_h\|_{\div}+\|\bfu-\bfu_h\|_1\right)\\
&\le C\left(\|\bfs-\bftau_h\|_{\div}+\|\bfp-\bfv_h\|_1\right)\|\bff\|_0.
\endaligned
\]
We observe explicitly that we cannot obtain any rate of convergence out of the
term $\|\bfsigma-\bfsigma_h\|_{\div}$ since $\div\bfsigma$ is only in
$L^2(\Omega)$ if $\bff$ is not more regular. On the other hand, the required
uniform convergence comes from the (minimal) regularity of the dual
problem~\eqref{eq:dualF} so that we get
\[
\|\bfu-\bfu_h\|_0^2\le Ch^s\|\bfu-\bfu_h\|_0\|\bff\|_0
\]
for some positive $s$ as long as the finite element spaces $\Sigma_h$ and
$U_h$ satisfy the following approximation properties
\begin{equation}
\aligned
&\inf_{\bftau\in\Sigma_h}\|\bfs-\bftau\|_{\div}\le
Ch^s\|\bfs\|_s+\|\div\bfs\|_s\\
&\inf_{\bfv\in U_h}\|\bfp-\bfv\|_1\le Ch^s\|\bfp\|_{1+s}.
\endaligned
\label{eq:approx}
\end{equation}

\begin{remark}
The power $s$ appearing in the estimate of $\|\bfu-\bfu_h\|_0$, which is
limited by $\bff\in L^2(\Omega)^d$, is not related to the rate of convergence
of the numerical scheme, but guarantees the uniform convergence that implies
the correct approximation of the spectrum. The optimal convergence of the
scheme is shown in the next theorem by using the full regularity of the
eigenspaces.
\label{re:uniform}
\end{remark}

\begin{theorem}
Let $\omega_i=\omega_{i+1}=\dots=\omega_{i+m-1}$ be an eigenvalue
of~\eqref{eq:vareigF} and denote by $E=\bigoplus\limits_{j=i}^{i+m-1}E_j$ the
corresponding eigenspace. If the discrete spaces $\Sigma_h$ and $U_h$ satisfy
the approximation properties~\eqref{eq:approx} then for $h$ small enough (so
that $\dim U_h\ge i+m-1$) the $m$ discrete eigenvalues
$\omega_{i,h}=\omega_{i+1,h}=\dots=\omega_{i+m-1,h}$ of~\eqref{eq:vareigFh}
converge to $\omega_i$; let us denote by $E_h$ the sum of the discrete
eigenspaces. Then the following error estimates hold true
\[
\aligned
&\delta(E,E_h)\le C\rho(h)\\
&|\omega_j-\omega_{j,h}|\le\rho(h)^2&&\qquad j=i,\dots,i+m-1
\endaligned
\]
with
\[
\rho(h)=
\sup_{\substack{\bfu\in E\\\|\bfu\|=1}}\inf_{\substack{\bftau\in\Sigma_h\\\bfv\in U_h}}
\left(\|\symgrad(\bfu)-\bftau\|_{\div}+\|\bfu-\bfv\|_1\right).
\]
\label{th:1}
\end{theorem}

\begin{proof}
The proof is standard (see~\cite{ls4eig,acta}); the only delicate part
consists in showing the double order of convergence for the eigenvalues, since
the discrete problem is not symmetric. The result can be obtained by
considering the adjoint problem, performing the corresponding analysis for
its approximation (with $\rho(h)$ replaced by $\rho^*(h)$) and by using the
standard Babu\v ska--Osborn theory~\cite{BaOs} from which we can conclude that
the error of the eigenvalues is bounded by $\rho(h)\rho^*(h)$ (note that the
continuous eigenvalues have ascent multiplicity equal to one since the
continuous problem is symmetric).
\end{proof}

\subsection{$L^2(\Omega)$-estimate for the three-field formulation}

The $L^2(\Omega)$ error estimate for the discretization of the three-field
formulation~\eqref{eq:vareigG} has been presented
in~\cite[Theorem~3]{BertrandCaiPark2019} in the case of a convex domain so
that the $H^2$ regularity of a generalized Stokes problem could be used
(see~\cite{CaiKu}). Since we are not interested in optimal estimates, but only
in the uniform convergence in the spirit or Remark~\ref{re:uniform}, the
arguments of~\cite{CaiKu} and~\cite{BertrandCaiPark2019} could be adapted in
order to get the estimate
\[
\|\bfu-\bfu_h\|_0\le Ch^s\|\bff\|_0
\]
where $\bfu$ solves~\eqref{eq:varsourceG} and $\bfu_h$ is the corresponding
discrete solution, provided the following approximation properties are
satisfied
\begin{equation}
\aligned
&\inf_{\bftau\in\Sigma_h}\|\bfs-\bftau\|_{\div}\le
Ch^s\|\bfs\|_s+\|\div\bfs\|_s\\
&\inf_{\bfv\in U_h}\|\bfp-\bfv\|_1\le Ch^s\|\bfp\|_{1+s}\\
&\inf_{\bfphi\in\Phi_h}\|\bfpsi-\bfphi\|_0\le Ch^s\|\bfpsi\|_s.
\endaligned
\label{eq:approx3}
\end{equation}

This allows to state the following convergence theorem
which is the analogous of Theorem~\ref{th:1} in this situation.
In analogy to~\eqref{eq:approx} we make the requirements on the
approximation properties of our spaces explicit.

\begin{theorem}
Let $\omega_i=\omega_{i+1}=\dots=\omega_{i+m-1}$ be an eigenvalue
of~\eqref{eq:vareigG} and denote by $E=\bigoplus\limits_{j=i}^{i+m-1}E_j$ the
corresponding eigenspace. If the discrete spaces $\Sigma_h$, $U_h$, and
$\Phi_h$ satisfy the approximation properties~\eqref{eq:approx3} then for $h$
small enough (so that $\dim U_h\ge i+m-1$) the $m$ discrete eigenvalues
$\omega_{i,h}=\omega_{i+1,h}=\dots=\omega_{i+m-1,h}$ of~\eqref{eq:vareigGh}
converge to $\omega_i$; let us denote by $E_h$ the sum of the discrete
eigenspaces. Then the following error estimates hold true
\[
\aligned
&\delta(E,E_h)\le C\rho(h)\\
&|\omega_j-\omega_{j,h}|\le\rho(h)^2&&\qquad j=i,\dots,i+m-1
\endaligned
\]
with
\[
\rho(h)=
\sup_{\substack{\bfu\in E\\\|\bfu\|=1}}
\inf_{\substack{\bftau\in\Sigma_h\\\bfv\in U_h\\\bfphi\in\Phi_h}}
\left(\|\grad(\bfu)-\bftau\|_{\div}+\|\bfu-\bfv\|_1+\|\curl(\bfu)-\bfphi\|_0\right).
\]
\label{th:2}
\end{theorem}

\section{Numerical results}
\label{se:num}

Our numerical results confirm the theoretical investigations of the previous
sections.

The numerical solution of a generalized eigenvalue problem such as those
arising from the discretization of~\eqref{eq:vareigF} and~\eqref{eq:vareigG}
can be challenging and in this paper we do not focus on the efficiency of the
solver but rather on the accuracy of the obtained results.

\DB
More specifically, we have to solve a nonsymmetric generalized eigenvalue
problem with several infinite eigenvalues. In our computations, we followed
two main strategies and we compared the two in order to make sure that no
artifact was introduced by the numerical method. We assembled the matrices in
FEniCS~\cite{fenics}. Then our first approach is to solve the eigenvalue
problem with SLEPc~\cite{Hernandez:2005:SSF} by reversing the matrix on the
left-hand side and the one on the right-hand side (see Case~(4) at the
beginning of Section~\ref{se:numap}); we used as options
``\texttt{shift-and-invert}'' with target ``\texttt{magnitude}''. The second
approach consists in exporting the matrices to Matlab and solve with the
built-in command ``\texttt{eigs}''.

More advanced numerical experiments are in progress, which will assess other
solvers and compare their performances.
\BD

\subsection{Numerical results on the square}
We start by taking the domain equal to the unit square $\Omega=]0,1[^2$. In
this case a pretty accurate estimate of the first eigenvalue is given by
$\omega=52.344691168$ if we consider
\[
\cA\bftau=\frac12\left(\bftau-\frac12\tr(\bftau)\Id\right),
\]
which corresponds to a Stokes problem, and homogeneous Dirichlet boundary
conditions everywhere, that is $\Gamma_N=\emptyset$ (see, for
instance,~\cite{GedickeKhan}).
Since we also want to investigate the symmetry of the formulation, we consider
three different mesh sequences: a symmetric and uniform mesh (CROSSED), a
non-symmetric and uniform mesh (RIGHT), and a non structured non symmetric
mesh (NONUNIF).
The meshes are indexed with respect to the number $N$ of subdivisions of the
square's sides.  The three meshes for $N=4$ are plotted in
Figure~\ref{fg:meshes}.

\begin{figure}
\includegraphics[width=4.1cm]{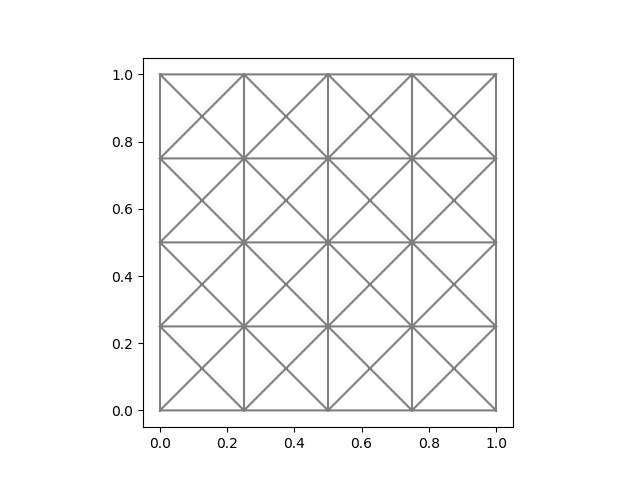}
\includegraphics[width=4.1cm]{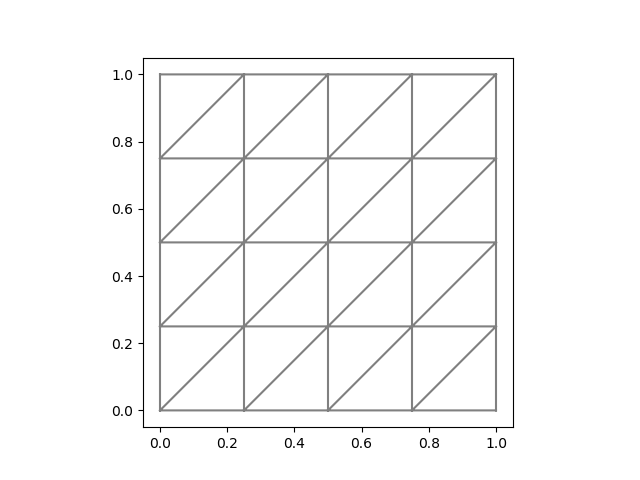}
\includegraphics[width=4.1cm]{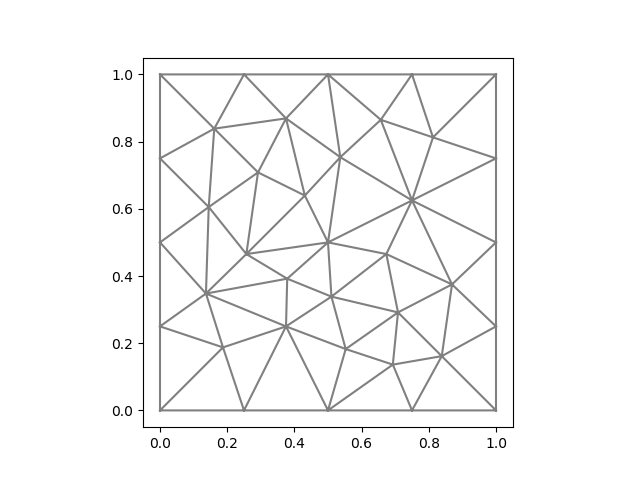}
\caption{Meshes on the unit square\label{fg:meshes}}
\end{figure}

In Table~\ref{tb:squareF} we report the numerical results corresponding to
computations performed with the two-field scheme~\eqref{eq:vareigFh} when
using the three mesh sequences with the level of refinement varying from $N=4$
to $N=12$. We considered a second order scheme based on Raviart--Thomas
elements for the approximation of $\Sigma_h$ (satisfying
assumption~\eqref{eq:approx}). The estimate of Theorem~\ref{th:1} predicts
fourth order of convergence for the eigenvalues which is confirmed by our
tests. It is interesting to observe that the fist computed eigenvalue is
always real in our tests.

\begin{table}
\footnotesize
\caption{Eigenvalues of the two-field formulation on the unit
square: $\Sigma_h=RT_1$, $U_h=P_2^c$\label{tb:squareF}}
\begin{center}
\begin{tabular}{l|r|r|r|r|r}
Mesh&$N=4$&$N=6$ (rate)&$N=8$ (rate)&$N=10$ (rate)&$N=12$ (rate)\\
\hline
CROSSED& 52.618734 &  52.400609 (3.9) &  52.362201 (4.0) &  52.351749 (4.1) &  52.348048 (4.1)\\
RIGHT& 54.132943 &  52.751624 (3.7) &  52.480276 (3.8) &  52.401472 (3.9) &  52.372369 (3.9)\\
NONUNIF& 52.744298 &  52.435687 (3.6) &  52.368139 (4.7) &  52.354017 (4.1) &  52.349733 (3.4)\\
\end{tabular}
\end{center}
\end{table}

We performed the same computations for the three-field formulation presented
in~\eqref{eq:vareigGh}. Also in this case we use a second order scheme based
on Raviart--Thomas elements (satisfying assumption~\eqref{eq:approx3}).
The corresponding results are reported in Table~\ref{tb:squareG}.
It turns out that also in this case the first computed eigenvalue is always
real and that the convergence properties match the theoretical results
with the exception of the convergence rate for $N=12$ on the non-uniform mesh.
In order to check better this phenomenon, we computed one more refinement
which is reported in Table~\ref{tb:square-nonunifF}. It is clear that the
overall convergence matches the expected fourth order rate and that the
non uniform behavior is due to the fact that the mesh sequence is not
structured.

\begin{table}
\footnotesize
\caption{Eigenvalues of the three-field formulation on the unit
square: $\Sigma_h=RT_1$, $U_h=P_2^c$, $\Phi_h=P_1^d$\label{tb:squareG}}
\begin{center}
\begin{tabular}{l|r|r|r|r|r}
Mesh&$N=4$&$N=6$ (rate)&$N=8$ (rate)&$N=10$ (rate)&$N=12$ (rate)\\
\hline
CROSSED& 52.523637 &  52.377459 (4.2) &  52.353859 (4.4) &  52.348025 (4.5) &  52.346144 (4.6)\\
RIGHT& 53.712947 &  52.621373 (3.9) &  52.426543 (4.2) &  52.375437 (4.4) &  52.358317 (4.5)\\
NONUNIF& 52.630390 &  52.398013 (4.1) &  52.355912 (5.4) &  52.348310 (5.1) &  52.347239 (1.9)
\end{tabular}
\end{center}
\end{table}

\begin{table}
\footnotesize
\caption{Eigenvalues of the three-field formulation on the unit
square: one more refinement for the non uniform mesh\label{tb:square-nonunifF}}
\begin{center}
\begin{tabular}{l|r|r|r|r|r}
Mesh&$N=6$ &$N=8$ (rate)&$N=10$ (rate)&$N=12$ (rate)&$N=14$ (rate)\\
\hline
NONUNIF& 52.398013 &  52.355912 (3.8) &  52.348310 (3.9) &  52.347239 (1.6) &  52.345561 (5.9)
\end{tabular}
\end{center}
\end{table}

Since the eigenvalue problems corresponding to the considered formulations are
not symmetric, it is interesting to investigate whether the computed
eigenvalues are complex or real. From the results that we are going show, it is
clear that in some cases the computed eigenvalues have a non-vanishing
imaginary part; this implies that in general the discrete formulations cannot
be reduced to symmetric problems. On the other hand, in most cases the
computed eigenvalues are real; this occurs, in particular on the symmetric
meshes.

In Table~\ref{tb:symmetric} we report the first five eigenvalues computed with
the three-field scheme on the non-uniform mesh for $N=10$ (similar results
could be presented for the two-field scheme as well). It is apparent that the
second and the third eigenvalue are complex conjugate.

\begin{table}
\footnotesize
\caption{First five eigenvalues computed with the three-field scheme on the
non-uniform mesh for $N=10$\label{tb:symmetric}}
\begin{center}
\begin{tabular}{l|r}
\#&Value\\
\hline
1&$52.348309870785620$\\
2&$92.163865261631784 - 0.000422328750065i$\\
3&$92.163865261631784 + 0.000422328750065i$\\
4&$128.2902654382040$\\
5&$154.3710922938166$
\end{tabular}
\end{center}
\end{table}

For the sake of comparison, Table~\ref{tb:nonsym} shows the same results on
the mesh for $N=11$ and we can see that in this case all the first five
eigenvalues are real. It is also interesting to observe that the second and
the third eigenvalues on the mesh for $N=10$ have the same real part, while
the corresponding eigenvalues on the mesh for $N=11$ are real and different
from each other. Clearly they are an approximation of the double eigenvalue
$\omega_2=\omega_3$.
In any case, this is perfectly compatible with the statements of
Theorems~\ref{th:1} and~\ref{th:2} where the difference
$|\omega_j-\omega_{j,h}|$ has to be understood in the sense of the distance in
the complex plane.

\begin{table}
\footnotesize
\caption{First five eigenvalues computed with the three-field scheme on the
non-uniform mesh for $N=11$\label{tb:nonsym}}
\begin{center}
\begin{tabular}{l|r}
\#&Value\\
\hline
1&52.346475661045424\\
2&92.147313995882541\\
3&92.151062887227738\\
4&128.2536615472890\\
5&154.2967136612170
\end{tabular}
\end{center}
\end{table}

\subsection{Numerical results on the L-shaped domain}

We conclude our numerical tests with computations on the L-shaped domain where
the re-entrant corner causes singularities in the solution. We consider a
uniform mesh (see Figure~\ref{fg:Lshaped} for the case $N=4$).

\begin{figure}
\includegraphics[width=4.1cm]{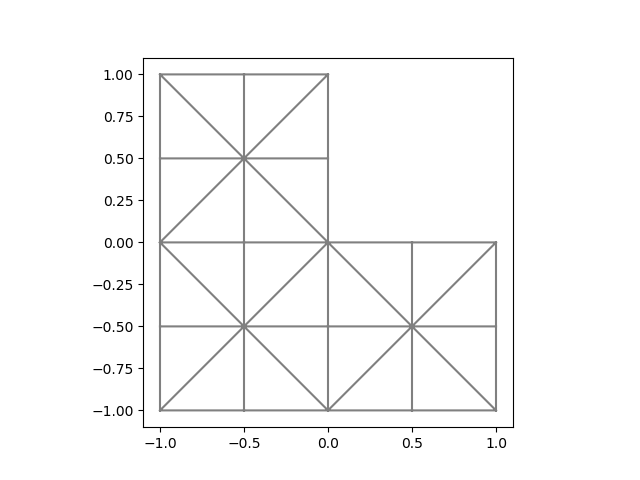}
\caption{Mesh of the L-shaped domain\label{fg:Lshaped}}
\end{figure}

An estimate of the first eigenvalues in this case has been provided
in~\cite{ruima} and corresponds to $\omega=32.13269464746$.
Tables~\ref{tb:LshapedF} and~\ref{tb:LshapedG} report on the computed
approximation with the two- and three-field
approximations~\eqref{eq:vareigFh} and~\eqref{eq:vareigGh}, respectively. As
expected, the singularity of the eigenfunction corresponding to the first
eigenvalue causes a reduction of the order of convergence.

\begin{table}
\footnotesize
\caption{Eigenvalues of the two-field formulation on the L-shaped domain:
$\Sigma_h=RT_1$, $U_h=P_2^c$\label{tb:LshapedF}}
\begin{center}
\begin{tabular}{l|r|r|r|r}
Mesh&$N=4$&$N=8$ (rate)&$N=16$ (rate)&$N=32$ (rate)\\
\hline
Uniform& 35.606285 &  31.937374 (4.2) &  31.871123 (-0.4) &  31.983573 (0.8)
\end{tabular}
\end{center}
\end{table}

\begin{table}
\footnotesize
\caption{Eigenvalues of the three-field formulation on the L-shaped domain:
$\Sigma_h=RT_1$, $U_h=P_2^c$, $\Phi_h=P_1^d$\label{tb:LshapedG}}
\begin{center}
\begin{tabular}{l|r|r|r|r}
Mesh&$N=4$&$N=8$ (rate)&$N=16$ (rate)&$N=32$ (rate)\\
\hline
Uniform& 34.132843 &  31.491151 (1.6) &  31.677105 (0.5) &  31.888816 (0.9)
\end{tabular}
\end{center}
\end{table}

Also in this case the first computed eigenvalue is real, however, also in
presence of a symmetric mesh, it turns out that there might be complex
eigenvalues. For instance, Table~\ref{tb:symL} reports some higher
eigenvalues computed for $N=8$. Also in this case the complex eigenvalues
converge towards real number according to the statement of
Theorems~\ref{th:1} and~\ref{th:2}; moreover, the presence of complex
eigenvalues depends on the mesh and on the level of refinements. The effect of
the mesh on complex eigenvalues will be the object of further investigation.

\begin{table}
\footnotesize
\caption{Some eigenvalues computed with the three-field scheme on the L-shaped
domain for $N=8$\label{tb:symL}}
\begin{center}
\begin{tabular}{l|r}
\#&Value\\
\hline
38& 339.9524318713583\\
39& 346.0018703851194\\
40& $350.8454478342160 - 2.5574107928386i$\\
41& $350.8454478342160 + 2.5574107928386i$\\
42& 359.0078935078376\\
43& 378.8779264703741
\end{tabular}
\end{center}
\end{table}

\section*{Acknowledgments}

The first author gratefully acknowledges support by the Deutsche
Forschungsgemeinschaft in the Priority Program SPP 1748 \textit{Reliable
simulation techniques in solid mechanics, Development of non standard
discretization methods, mechanical and mathematical analysis} under the
project number BE 6511/1-1.

%\bibliographystyle{amsplain}
%\bibliography{ref}

\begin{thebibliography}{10}

\bibitem{fenics}
\emph{Fenics project}, \url{https://fenicsproject.org/}.

\bibitem{vecquad}
Douglas~N. Arnold, Daniele Boffi, and Richard~S. Falk, \emph{Quadrilateral
  {$H({\rm div})$} finite elements}, SIAM J. Numer. Anal. \textbf{42} (2005),
  no.~6, 2429--2451.

\bibitem{BaOs}
Ivo Babu{\v{s}}ka and John Osborn, \emph{Eigenvalue problems}, Handbook of
  numerical analysis, {V}ol.\ {II}, Handb. Numer. Anal., II, North-Holland,
  Amsterdam, 1991, pp.~641--787.

\bibitem{ls4eig}
Fleurianne Bertrand and Daniele Boffi, \emph{First order least-squares
  formulations for eigenvalue problems}, arXiv:2002.08145 [math.NA], 2020.

\bibitem{ruima}
Fleurianne Bertrand, Daniele Boffi, and Rui Ma, \emph{An adaptive finite
  element scheme for the {H}ellinger--{R}eissner elasticity mixed eigenvalue
  problem}, submitted, 2020.

\bibitem{BertrandCaiPark2019}
Fleurianne Bertrand, Zhiqiang Cai, and Eun~Young Park, \emph{Least-squares
  methods for elasticity and {S}tokes equations with weakly imposed symmetry},
  Comput. Methods Appl. Math. \textbf{19} (2019), no.~3, 415--430.

\bibitem{acta}
Daniele Boffi, \emph{Finite element approximation of eigenvalue problems}, Acta
  Numer. \textbf{19} (2010), 1--120.

\bibitem{CaiKu}
Zhiqiang Cai and Jaeun Ku, \emph{The {$L^2$} norm error estimates for the div
  least-squares method}, SIAM J. Numer. Anal. \textbf{44} (2006), no.~4,
  1721--1734.

\bibitem{CaiStarke2004}
Zhiqiang Cai and Gerhard Starke, \emph{Least-squares methods for linear
  elasticity}, SIAM J. Numer. Anal. \textbf{42} (2004), no.~2, 826--842.

\bibitem{ref1}
K.~Andrew Cliffe, Tony~J. Garratt, and Alastair Spence, \emph{Eigenvalues of
  block matrices arising from problems in fluid mechanics}, SIAM J. Matrix
  Anal. Appl. \textbf{15} (1994), no.~4, 1310--1318.

\bibitem{GedickeKhan}
Joscha Gedicke and Arbaz Khan, \emph{Arnold-{W}inther mixed finite elements for
  {S}tokes eigenvalue problems}, SIAM J. Sci. Comput. \textbf{40} (2018),
  no.~5, A3449--A3469.

\bibitem{Hernandez:2005:SSF}
Vicente Hernandez, Jose~E. Roman, and Vicente Vidal, \emph{{SLEPc}: A scalable
  and flexible toolkit for the solution of eigenvalue problems}, {ACM} Trans.
  Math. Software \textbf{31} (2005), no.~3, 351--362.

\end{thebibliography}

\providecommand{\bysame}{\leavevmode\hbox to3em{\hrulefill}\thinspace}
\providecommand{\MR}{\relax\ifhmode\unskip\space\fi MR }
\providecommand{\MRhref}[2]{%
  \href{http://www.ams.org/mathscinet-getitem?mr=#1}{#2}
}
\providecommand{\href}[2]{#2}

\end{document}